\documentclass[12pt]{amsart}
\usepackage{mycommands}
\usepackage{mathrsfs}
\usepackage{tikz-cd}

\begin{document}

\title[On The Cohomology of \(N_C(-2)\) in Positive Characteristic]{\boldmath On The Cohomology of \(N_C(-2)\) in Positive Characteristic}
\maketitle

\begin{abstract}
Let \(C \subset \pp^3\) be a general Brill--Noether curve.
A classical problem is to determine when \(H^0(N_C(-2)) = 0\),
which controls the quadric section of \(C\).

So far this problem has only been solved in characteristic zero,
in which case \(H^0(N_C(-2)) = 0\) with finitely many exceptions.
In this note, we extend these results to positive characteristic,
uncovering a wealth of new exceptions in characteristic \(2\).
\end{abstract}

\section{Introduction}

Let \(C\) be a general curve of genus \(g\),
equipped with a general embedding \(C \subset \pp^3\) of degree~\(d\).
A classical problem is to determine the cohomology groups of twists of the normal bundle \(N_C\),
which control how \(C\) intersects surfaces.
Since \(\chi(N_C(-2)) = 0\), the most interesting case is the following.

\begin{quest*}
When is \(H^0(N_C(-2)) = 0\)?
\end{quest*}

Since the normal bundle controls the deformation theory of \(C\),
this question is closely linked to how \(C\) intersects a fixed quadric surface \(Q\).
More precisely, an affirmative answer to this question is equivalent
to the assertion that the map \([C] \dashrightarrow [Q \cap C]\) is generically \'etale.

This question was first studied by Ellingsrud and Hirschowitz \cite{eh},
and later by Perrin \cite{perrin},
who used liason to give a partial answer in characteristic zero.
Subsequent work of the author \cite{quadrics-orig}
determined that, in characteristic zero, \(H^0(N_C(-2)) = 0\) apart from six exceptions:
\[(d, g) \in \{(4, 1), (5, 2), (6, 2), (6, 4), (7, 5), (8, 6)\}.\]
These results ultimately found application in the proof of the maximal rank theorem
in characteristic zero; see \cite{mrc}.

On the other hand, in characteristic \(2\),
the normal bundle is the twist of a Frobenius pullback.
This has consequences for closely-related properties like stability \cite{stab-p3}
and interpolation \cite{interpolation},
which must therefore fail for rational space curves of even degree in characteristic \(2\).
The natural guess might thus be that this is the only additional reason
for the vanishing of \(H^0(N_C(-2))\) to fail in positive characteristic,
or in other words, that \(H^0(N_C(-2)) = 0\) except if:
\begin{itemize}
\item \((d, g) \in \{(4, 1), (5, 2), (6, 2), (6, 4), (7, 5), (8, 6)\}\) or
\item \(g = 0\) and \(d\) is even and the characteristic is \(2\).
\end{itemize}
Surprisingly, we show that this expectation is false.
In other words, there are additional cases where
\(H^0(N_C(-2)) \neq 0\) in characteristic \(2\),
corresponding to additional structure besides
merely the fact that \(N_C\) is the twist of a Frobenius pullback!
To state our theorem, we first make the following definition.

\begin{defi} A stable map \(f \colon C \to \pp^3\) is called
a \emph{Brill--Noether curve (BN-curve)} if it corresponds to a point in a
component of \(\bar{M}_g(\pp^r, d)\) which both dominates \(\bar{M}_g\),
and whose generic member is a nondegenerate map from a smooth curve.
Write
\[h(d, g) = h^0(N_{C(d, g)}(-2)) \quad \text{where} \quad C(d, g) \ \text{is a general BN-curve of degree \(d\) and genus \(g\)}.\]
\end{defi}

\begin{thm} \label{thm:main} We have
\begin{equation} \label{f-value}
h(d, g) = \begin{cases}
1 & \text{if the characteristic is \(2\) and \(d + g\) is even} \\
0 & \text{otherwise},
\end{cases}
\end{equation}
except in the following six exceptional cases:
\[\begin{array}{c|c}
(d, g) & h(d, g) \\\hline
(4, 1) & 2 \\
(5, 2) & 2 \\
(6, 2) & 1 \\
(6, 4) & 5 \\
(7, 5) & 3 \\
(8, 6) & 1
\end{array}\]
\end{thm}

We begin, in Section~\ref{lower2}, by describing the novel additional structure
that constrains the parity of \(h(d, g)\) in characteristic \(2\).
Sections~\ref{sec:six} and~\ref{sec:deg} review the arguments of \cite{quadrics-orig},
and indicate how they can be modified in positive characteristic.
In particular, we show that the proof of Theorem~\ref{thm:main} can be reduced to just four base cases.
These final four cases require more delicate arguments,
which occupy Sections~\ref{sec:63-75-86}--\ref{sec:74}.

As a byproduct of our methods, we also prove the following theorem,
which answers an analogous question for the twist of the normal bundle by \(-1\) in arbitrary characteristic.
To state the theorem, recall that a vector bundle \(\mathcal{E}\) on a curve \(C\)
is said to \textsf{satisfy interpolation} if \(H^1(\mathcal{E}) = 0\), and for a general effective divisor \(D\)
of any degree, either \(H^0(\mathcal{E}) = 0\) or \(H^1(\mathcal{E}) = 0\).

\begin{thm} \label{thm:one} For \(C(d, g)\) a general BN-curve of degree \(d\) and genus \(g\),
the bundle \(N_{C(d, g)}(-1)\) satisfies interpolation,
except if \((d, g) \in \{(5, 2), (6, 4)\}\),
or if \(g = 0\) and \(d\) is even and the characteristic is \(2\).
\end{thm}

\subsection*{Acknowledgements}
The author would like to thank Atanas Atanasov, Izzet Coskun, Isabel Vogt, and David Yang,
for many helpful conversations about normal bundles of curves.
This work was supported by NSF grant DMS-2200641, as well as NSF grant 1440140
while the author was in residence at MSRI/SLMath in Berkeley CA during the spring of 2023.

\section{Lower Bounds in Characteristic \(2\)\label{lower2}}

In this section, we explain the exceptional geometry in characteristic \(2\).
Our main result is Corollary~\ref{cor:parity} below, which establishes that
\(h(d, g) \equiv d + g + 1\) mod \(2\) in characteristic \(2\).

\subsection{The Exact Sequence of Projection}
In any characteristic, we have the Euler sequence for the conormal bundle:
\begin{equation} \label{euler}
0 \to N_C^\vee(1) \to \oo_C^4 \to \mathscr{P}^1(\oo_C(1)) \to 0,
\end{equation}
where \(\mathscr{P}^1(\oo_C(1))\) denotes the bundle of first principal parts of \(\oo_C(1)\).
For a general choice of \(\oo_C\) quotient in the middle term (corresponding to a general point in \(\pp^3\)),
the Euler sequence induces a map \(N_C^\vee(1) \to \oo_C\).
We therefore obtain the exact sequence:
\begin{equation} \label{point-dual}
0 \to \wedge^2 \mathscr{P}^1(\oo_C(1))^\vee \simeq K_C^\vee(-2) \to N_C^\vee(1) \to \oo_C \to 0,
\end{equation}
and a corresponding extension class
\[e \in \ext^1(\oo_C, K_C^\vee(-2)) \simeq H^1(K_C^\vee(-2)).\]

Dualizing and twisting, this gives rise to the normal bundle sequence induced
by projection from the point in \(\pp^3\) corresponding to our choice of \(\oo_C\) quotient:
\[0 \to \oo_C(-1) \to N_C(-2) \to K_C(1) \to 0.\]
Since \(H^0(\oo_C(-1)) = H^1(K_C(1)) = 0\) for degree reasons,
the associated long exact sequence in cohomology
implies that our desired cohomology groups
\(H^0(N_C(-2))\) and \(H^1(N_C(-2))\) are the kernel
and cokernel respectively of the boundary map
\[H^0(K_C(1)) \to H^1(\oo_C(-1)).\]
Using Serre duality between \(H^0(K_C(1))\) and \(H^1(\oo_C(-1))\),
this boundary map may be regarded as a bilinear form
\[\delta \colon H^0(K_C(1)) \times H^0(K_C(1)) \to H^1(K_C) \simeq k,\]
obtained by multiplying sections and taking the cup product
with the extension class \(e\).
Our goal in the following two subsections is to prove Proposition~\ref{prop:skew},
which asserts that \(\delta\) is skew-symmetric in characteristic \(2\) --- and thus has even rank.

\subsection{The Frobenius Morphism and Its Friends}
Here we recall some standard constructions in positive characteristic;
for ease of notation, we will suppose the characteristic is \(2\).
For a more detailed discussion, the reader can consult
\cite[\S2.1--2.6]{cartier-qr}, \cite[\S4]{raynaud}, and/or \cite[\S10]{serre}.

Write \(F \colon C \to C'\) for the relative Frobenius morphism.
If \(L\) is a line bundle on \(C\) (respectively \(L'\) is a line bundle on \(C'\)),
then the norm and squaring maps are linear:
\begin{align*}
\nm_* \colon H^0(C, L) &\to H^0(C', \nm L) \\
\sq \colon H^0(C, L) &\to H^0(C, L^{\otimes 2}) \\
\sq' \colon H^0(C', L') &\to H^0(C', (L')^{\otimes 2}).
\end{align*}
By construction, \(\sq = F^* \circ \nm_*\)
and \(\sq' = \nm_* \circ F^*\).

Recall that, if \(x\) denotes a local coordinate on \(C\), so that \(y = x^2\) gives a local coordinate on \(C'\),
then the \emph{Cartier operator}
\[c \colon F_* K_C \to K_{C'}\]
is, in our case, given by the formula
\[c\big((a_0 + a_1 x + a_2 x^2 + a_3 x^3 + \cdots) \cdot dx\big) = (\sqrt{a_1} + \sqrt{a_3} y + \sqrt{a_5} y^2 + \sqrt{a_7} y^3 + \cdots) \cdot dy.\]
The Cartier operator is independent of choice of local coordinate \(x\), and 
can be defined more generally on any smooth scheme \(X\)
in any positive characteristic \(p\),
as an operator from closed \(i\)-forms on \(X\) to \(i\)-forms on \(X'\).
For details see \cite[\S2.6]{cartier-qr}.

Finally, let \(B\) denote the \emph{sheaf of locally exact differentials},
which is a sheaf of \(\oo_{C'}\)-modules that can be defined in several equivalent manners:
\begin{enumerate}
\item As the sheafification of the presheaf on \(C'\) whose value on an open set \(U\)
is the set of exact differentials on \(F^{-1}(U)\).
\item As those differentials of the form \((a_0 + a_2 x^2 + a_4 x^4 + \cdots) \cdot dx\),
for some (or equivalently for any) local coordinate \(x\) on \(C\).
\item As the cokernel of the adjoint map \(\oo_{C'} \to F_* \oo_C\).
\item As the kernel of the Cartier operator (recalled above) \(c \colon F_* K_C \to K_{C'}\).
\end{enumerate}
It is well-known that the sheaf \(B\) is a square root of the canonical bundle, in both senses:
\[F^* B \simeq K_C \quad \text{and} \quad B^{\otimes 2} \simeq K_{C'}.\]
Indeed, the first of these isomorphisms is induced by pullback of differentials,
while the second is induced by the map \((F_* \oo_C)^{\otimes 2} \to K_{C'}\)
given by \(f \otimes g \mapsto c(f \cdot dg) = -c(df \cdot g) = c(df \cdot g)\).
In particular, \(K_{C'} \simeq \nm K_C\).

\subsection{\boldmath The Bilinear Form \(\delta\) in Characteristic \(2\)} Here we prove the following.

\begin{prop} \label{prop:skew}
Suppose that the characteristic is \(2\).
Then the bilinear form \(\delta\) is skew-symmetric, i.e., \(\delta(w, w) = 0\) for any \(w \in H^0(K_C(1))\).
\end{prop}
\begin{proof}
Our first claim is that the extension class \(e\) lies in the image of the pullback map
\[F^* \colon H^1(B^\vee \otimes \nm \oo_C(-1)) \to H^1(K_C^\vee(-2)).\]
Indeed, since \(\mathscr{P}^1(\oo_C(1)) \simeq F^* F_* \oo_C(1)\),
and the evaluation map \(\oo_C^4 \to F^* F_* \oo_C(1)\) is the pullback under Frobenius
of the evaluation map \(\oo_{C'}^4 \to F_* \oo_C(1)\),
the entire Euler sequence \eqref{euler} is the pullback of an exact sequence under Frobenius.
Moreover, the formation of \eqref{point-dual} from \eqref{euler}
is also compatible with Frobenius. More precisely, \eqref{point-dual} is the pullback under Frobenius of
an exact sequence of the form
\[0 \to \wedge^2 (F_* \oo_C(1))^\vee \simeq B^\vee \otimes \nm \oo_C(-1) \to \bullet \to \oo_{C'} \to 0.\]
This implies that \(e\) is a pullback under Frobenius as desired.

Next, we claim that
the image of \(\sq \colon H^0(K_C(1)) \to H^0(K_C^{\otimes 2}(2))\)
lies in the kernel of the Cartier operator. This follows from the following commutative diagram
(because the composition along the bottom row is zero):
\[\begin{tikzcd}[column sep=small]
H^0(K_C(1)) \arrow[r, "\sq"] \arrow[d, "\nm_*"] & H^0(K_C \otimes K_C(2)) \arrow[d, equal, "\text{push-pull}"]\\
H^0(B \otimes B \otimes \nm \oo_C(1)) \arrow[r] & H^0(F_* K_C \otimes B \otimes \nm \oo_C(1)) \arrow[r, "c_*"] &  H^0(K_{C'} \otimes B \otimes \nm \oo_C(1))
\end{tikzcd}\]

Since the Cartier operator is the Serre dual of
Frobenius pullback (see for example \cite[\S10]{serre}),
it follows that \(\delta(w, w) = 0\) as desired.
\end{proof}

\begin{cor} \label{cor:parity} In characteristic \(2\), we have \(h(d, g) \equiv d + g + 1\) mod \(2\).
\end{cor}
\begin{proof}
This follows from Proposition~\ref{prop:skew},
since the rank of a skew-symmetric form is even,
and \(h^0(K_C(1)) = d + g - 1\).
\end{proof}

\section{Review: The Six Exceptional Cases \label{sec:six}}

The geometric descriptions given in \cite{quadrics-orig}
quickly yield that \(h(d, g)\)
is at least the value claimed by Theorem~\ref{thm:main}
in the six exceptional cases
(and in fact is equal to the claimed values when
\((d, g) \in \{(4, 1), (5, 2), (6, 4)\}\)).
For completeness, we briefly recall these descriptions here.

\subsection{\boldmath \((d, g) = (4, 1)\)}
Such curves are the complete intersection of two quadrics,
so we have \(N_C \simeq \oo_C(2)^2\). In particular \(h^0(N_C(-2)) = 2\).

\subsection{\boldmath \((d, g) = (5, 2)\)}
Such curves lie on a quadric \(Q\). By a Chern class computation, we have \(N_{C/Q} \simeq K_C(2)\), and so
\(h^0(N_C(-2)) \geq h^0(N_{C/Q}(-2)) = h^0(K_C) = 2\).
In fact, with a bit more work, one can show \(h^0(N_C(-2)) = 2\) (see \cite[Lemma 3.1]{stab-p3}).

\subsection{\boldmath \((d, g) = (6, 2)\)}
Such curves are the projection of a curve \(\tilde{C} \subset \pp^4\) from a point not lying on \(\tilde{C}\),
and \(\tilde{C} = Q \cap S\) is the intersection of a quadric hypersurface \(Q\) and a cubic scroll \(S\) in \(\pp^4\).
In particular \(h^0(N_C(-2)) \geq h^0(N_{\tilde{C}/S}(-2)) = h^0(\oo_{\tilde{C}}) = 1\).

\subsection{\boldmath \((d, g) = (6, 4)\)}
Such curves are the complete intersection of a quadric and cubic surface,
so we have \(N_C \simeq \oo_C(2) \oplus \oo_C(3)\).
In particular \(h^0(N_C(-2)) = 5\).

\subsection{\boldmath \((d, g) = (7, 5)\)}
Such curves are the projection of a canonical curve \(\tilde{C} \subset \pp^4\)
from a point \(p \in \tilde{C}\), and \(\tilde{C}\) is the complete intersection of three quadrics.
In particular, we have an exact sequence
\[0 \to \oo_{\tilde{C}}(1)(2p) \simeq \oo_C(1)(3p) \to \oo_{\tilde{C}}(2)^{\oplus 3} \simeq \oo_C(2)(2p)^{\oplus 3} \to N_C(p) \to 0.\]
Therefore \(h^0(N_C(-2)) \geq 3 \cdot h^0(\oo_C(p)) - h^0(\oo_C(-1)(2p)) = 3 \cdot 1 - 0 = 3\).

\subsection{\boldmath \((d, g) = (8, 6)\)}
Such curves lie on a cubic surface \(S\). By a Chern class computation, we have \(N_{C/S} \simeq K_C(1)\), and so
\(h^0(N_C(-2)) \geq h^0(N_{C/S}(-2)) = h^0(K_C(-1)) = 1\).

\section{Review: Degeneration Arguments \label{sec:deg}}

The basic strategy of \cite{quadrics-orig} to prove upper bounds is degeneration to reducible curves.
In this section, we review these arguments.
We indicate how trivial modifications can be made to remove
the characteristic zero assumption in all but four cases,
which will be take up in the following sections.
To show the reducible curves constructed in \cite{quadrics-orig} are BN-curves,
we will use \cite[Theorems 1.6 and 1.7]{rbn}.
Although \cite{rbn} assumes characteristic zero for the proofs of its main theorem,
this assumption does not enter into the proofs of these two theorems.

\begin{lm}[{Variant of \cite[Lemma 2.6]{quadrics-orig}}] \label{kh1}
Let \(f \colon C \cup_\Gamma D \to \pp^r\) be an unramified map from a reducible curve,
with \(C\) and \(D\) smooth,
and let \(E\) and \(F\) be
divisors supported on \(C \smallsetminus \Gamma\) and \(D \smallsetminus \Gamma\)
respectively. Write
\[\alpha \colon H^0(N_{f|_D}(-F)) \to \bigoplus_{p \in \Gamma} \left(\frac{T_p (\pp^r)}{f_* (T_p (C \cup_\Gamma D))}\right).\]
Then
\[h^0(N_f(-E-F)) \leq \dim \ker \alpha + \codim \big(H^0(N_{f|_D} (-F)) \subseteq H^0(N_f|_D (-F))\big) + h^0(N_{f|_C} (-E)).\]
\end{lm}
\begin{proof}
This follows as in \cite[Lemma 2.6]{quadrics-orig}, which states that
\(h^0(N_f(-E-F)) = 0\) provided that \(\alpha\) is injective,
\(H^0(N_{f|_D} (-F)) = H^0(N_f|_D (-F))\), and \(h^0(N_{f|_C} (-E)) = 0\).
\end{proof}

\begin{lm}[{Variant of \cite[Lemma 5.3]{quadrics-orig}}] \label{attach-can}
Let \(\Gamma \subset \pp^3\) be a set of \(5\) general points,
\(C\) a general BN-curve passing through \(\Gamma\),
and \(D\) a general canonical curve passing through \(\Gamma\).
Then \(h^0(N_{C \cup D}(-2)) \leq h^0(N_C(-2))\) and
interpolation for \(N_C(-1)\) implies interpolation for \(N_{C \cup D}(-1)\).
\end{lm}
\begin{proof}
Let \(E\) and \(F\) be divisors supported on \(C \smallsetminus \Gamma\) and \(D \smallsetminus \Gamma\)
with \(\oo_D(F) \simeq \oo_D(2)\).
The same argument as in \cite[Lemma 5.3]{quadrics-orig},
using Lemma~\ref{kh1} in place of \cite[Lemma 2.6]{quadrics-orig},
implies \(h^0(N_{C \cup D}(-E-F)) \leq h^0(N_C(-E))\).
The desired results follow.
\end{proof}

If \(C\) and \(D\) are as in Lemma~\ref{attach-can},
then by \cite[Theorem 1.6]{rbn}, the resulting curve \(C \cup D\)
is a BN-curve of degree \(d + 6\) and genus \(g + 8\).
We conclude that \eqref{f-value} for all general BN-curves of genus \(g\) implies
\eqref{f-value} for all general BN-curves of genus \(g + 8\), respectively interpolation for
\(N_C(-1)\) for all general BN-curves of genus \(g\) implies interpolation for \(N_C(-1)\)
for all general BN-curves of genus \(g + 8\). It therefore suffices
to prove Theorems~\ref{thm:main} and~\ref{thm:one} for
\[g \in \{0, 1, 2, 3, 4, 5, 6, 7, 9, 10, 12, 13, 14\},\]
plus Theorem~\ref{thm:one} for \(g = 8\).

For any such genus \(g\), the Brill--Noether theorem implies \(d \geq g\),
which is equivalent to inequality (b) in \cite[Proposition 4.12]{firstpaper}
for \(E = N_C(-1)\).
Combining \cite[Proposition 4.12]{firstpaper} with \cite[Theorem 1.4]{interpolation} therefore
completes the proof of Theorem~\ref{thm:one}.
For the remainder of the paper we thus consider only Theorem~\ref{thm:main}.

\begin{lm}[{Variant of \cite[Lemma 5.2]{quadrics-orig}}] \label{lm:addline}
Let \(C \subset \pp^3\) be a general BN-curve,
and \(L\) be a general \(1\)-secant line.
Then
\[h^0(N_{C \cup L}(-2)) \leq \begin{cases}
h^0(N_C(-2)) - 1 & \text{if \(h^0(N_C(-2)) > 0\);} \\
1 & \text{if \(h^0(N_C(-2)) = 0\) and the characteristic is \(2\);} \\
0 & \text{if \(h^0(N_C(-2)) = 0\) and the characteristic is zero or odd.}
\end{cases}\]
\end{lm}
\begin{proof}
The first two cases follow from Lemma~\ref{kh1}, with \((C, D) = (L, C)\).
The final case follows from \cite[Lemma 5.2]{quadrics-orig}
(whose proof works when the characteristic is zero or odd).
\end{proof}

We conclude that \eqref{f-value} for BN-curves of degree \(d\) and genus \(g\) implies
\eqref{f-value} for BN-curves of degree \(d + 1\) and genus \(g\),
and moreover that the truth of Theorem~\ref{thm:main} for \((d, g) = (5, 2)\) (respectively for \((d, g) = (8, 6)\))
implies the truth of Theorem~\ref{thm:main} for \(g = 2\) (respectively for \(g = 6\)).
This reduces the proof of Theorem~\ref{thm:main} to a finite number of cases:
\begin{multline*}
(d, g) \in \{(3, 0), (4, 1), (5, 1), (5, 2), (6, 3), (6, 4), (7, 4), (7, 5), (8, 5), \\
(8, 6), (9, 7), (10, 9), (11, 10), (12, 12), (13, 13), (14, 14)\}.
\end{multline*}
All but four of these cases follow either from trivial modifications of arguments in \cite{quadrics-orig},
or directly from the above results.

\subsection{\boldmath \((d, g) = (3, 0)\)} In this case, \(N_C\) is balanced by \cite[Theorem 1]{stab-p3},
so \(h(3, 0) = 0\) as desired.

\subsection{\boldmath \((d, g) \in \{(4, 1), (5, 2), (6, 4)\}\)}
These cases were already settled in Section~\ref{sec:six}.  

\subsection{\boldmath \((d, g) = (5, 1)\)}
When the characteristic is zero or odd,
the proof in \cite[Section 10]{quadrics-orig} applies to show \(h(5, 1) = 0\).
(The only reason this argument fails in characteristic \(2\) is because the curve \(f(L)\) appearing in the proof is a strange curve;
of course, this cannot happen when the characteristic is zero or odd.)

In characteristic \(2\), we may apply Lemma~\ref{lm:addline} to show \(h(5, 1) \leq 1\);
since \(h(5, 1)\) is odd by Corollary~\ref{cor:parity},
this shows \(h(5, 1) = 1\) as desired.

\subsection{\boldmath \((d, g) \in \{(10, 9), (11, 10), (12, 12), (13, 13), (14, 14)\}\)}
These cases follow as in \cite[Lemma 7.1]{quadrics-orig},
again using Lemma~\ref{kh1} in place of \cite[Lemma 2.6]{quadrics-orig}.

\subsection{\boldmath \((d, g) \in \{(6, 3), (9, 6), (9, 7)\}\)}
We construct reducible curves with \(H^0(N_C(-2)) = 0\) as in
\cite[Corollary 8.2]{quadrics-orig}.

To show the resulting curves are BN-curves,
we use \cite[Theorem 1.6]{rbn} for \((d, g) \in \{(6, 3), (9, 6)\}\),
respectively \cite[Theorem 1.7]{rbn} for \((d, g) = (9, 7)\).
(The original argument in \cite{quadrics-orig} shows the resulting curves are BN-curves
by using results of \cite{keem} that require a characteristic zero hypothesis.)

\subsection{\boldmath The Remaining Cases: \((d, g) \in \{(7, 4), (7, 5), (8, 5), (8, 6)\}\)}
In these cases, the arguments of \cite{quadrics-orig}
break down more seriously in positive characteristic,
thereby requiring new ideas.
The remainder of the paper will be devoted to these four cases.

\section{The Cases \((d, g) \in \{(7, 5), (8, 6)\}\) \label{sec:63-75-86}}

In these cases, we apply Lemma~\ref{kh1}. We take \(C\) to be a \(2\)-secant line,
respectively the union of two disjoint \(2\)-secant lines,
to a curve \(D\) of degree \(6\) and genus \(4\).

The curve \(D\) lies on a unique (smooth) quadric \(Q\).
The surjectivity of \(\alpha\) can be shown by restricting to the subspace
\(H^0(N_{D/Q}(-2))\), so \(\dim \ker \alpha\) is \(3\) and \(1\) respectively.
The equality \(H^0(N_D(-2)) = H^0(N_{C \cup D}|_D(-2))\) follows from
\(H^0(\oo_D(C \cap D)) = H^0(\oo_D)\), thanks to the following commutative diagram:
\[\begin{tikzcd}
0 \arrow[r] & N_{D/Q}(-2) \arrow[r] \arrow[d, equal] & N_D(-2) \arrow[r] \arrow[d] & N_Q|_D(-2) \simeq \oo_D \arrow[r] \arrow[d] & 0 \\
0 \arrow[r] & N_{D/Q}(-2) \arrow[r] & N_{C \cup D}|_D(-2) \arrow[r] & \oo_D(C \cap D) \arrow[r] & 0.
\end{tikzcd}\]
Since \(H^0(N_C(-2)) = 0\),
the upper bound from Lemma~\ref{kh1} is \(\dim \ker \alpha\), which
matches the lower bounds established in Section~\ref{sec:six}.

\section{The Case \((d, g) = (8, 5)\) \label{sec:85}}

We begin by taking a hyperelliptic curve \(D\) of genus \(3\),
and points \(p_1\) and \(p_2\) not conjugate under the hyperelliptic involution.
Write \(f \colon D \to \pp^3\) for the map obtained from the complete linear system \(|2H + p_1 + p_2|\).
By construction, \(f\) maps \(p_1\) and \(p_2\) to a common point \(q \in \pp^3\),
but is injective on every other divisor of degree \(2\)
(so in particular unramified).
Projection from \(q\) realizes the complete linear system \(|2H|\),
which maps \(2\)-to-\(1\) onto a plane conic;
therefore, the image of \(f\) lies on a singular quadric \(Q\) with vertex at \(q\).
Write \(N_{D \to Q}\) for the normal sheaf of the map \(D \to Q\).
Using the exact sequence
\[0 \to N_{D \to Q}(-2) \simeq K_D(p_1 + p_2) \to N_f(-2) \to N_Q|_D(-2) \simeq \oo_D(-p_1-p_2) \to 0,\]
we see that \(h^0(N_f(-2)) = 4\), with all sections arising from the subbundle \(N_{D \to Q}(-2)\).
Since \(6 > 2 \cdot 3 - 2\), the map \(f\) is automatically a BN-curve.

We attach general \(2\)-secant lines \(L_1\) and \(L_2\) to \(D\),
with \(L_i\) meeting \(D\) at points \(\{q_{i1}, q_{i2}\}\).
By \cite[Theorem 1.6]{rbn}, the resulting map \(\hat{f} \colon D \cup L_1 \cup L_2 \to \pp^3\)
is a BN-curve.
We then apply Lemma~\ref{kh1}.

The injectivity of \(\alpha\) follows from the generality of the \(q_{ij}\)
and the fact that \(h^0(N_f(-2)) = 4\) with all sections arising from the subbundle \(N_{D \to Q}(-2)\),
which is transverse to the \(L_i\).
The equality \(H^0(N_f(-2)) = H^0(N_{\hat{f}}|_D(-2))\) follows from
\[H^0(\oo_D(q_{11} + q_{12} + q_{21} + q_{22} -p_1-p_2)) = 0 = H^0(\oo_D(-p_1-p_2)),\]
thanks to the following commutative diagram:
\[\begin{tikzcd}
0 \arrow[r] & N_{D \to Q}(-2) \arrow[r] \arrow[d, equal] & N_f(-2) \arrow[r] \arrow[d] & \oo_D(-p_1-p_2) \arrow[r] \arrow[d] & 0 \\
0 \arrow[r] & N_{D \to Q}(-2) \arrow[r] & N_{\hat{f}}|_D(-2) \arrow[r] & \oo_D(q_{11} + q_{12} + q_{21} + q_{22} -p_1-p_2) \arrow[r] & 0.
\end{tikzcd}\]
Since \(H^0(N_{L_i}(-2)) = 0\), this completes the proof.

\section{The Case \((d, g) = (7, 4)\) \label{sec:74}}

In this section, we establish Theorem~\ref{thm:main} for \((d, g) = (7, 4)\).
Rather than computing the cohomology group \(H^0(N_C(-2))\),
we will reason geometrically to show that \([C] \dashrightarrow [C \cap Q]\) is generically \'etale.
In fact, our argument will show a bit more: This map is generically \'etale of degree \(3\),
and has a Galois group isomorphic to \(S_3\).

Let \(\Gamma \subset Q \simeq \pp^1 \times \pp^1\) be a general set of \(14\) points.
Then \(\Gamma\) lies on a pencil of \((3, 3)\)-curves,
and the residual to \(\Gamma\) in the base locus is a general set of \(4\) points \(\Gamma' \subset Q\).
After a generically \'etale basechange (of degree \(3\) with Galois group \(S_3\)),
we may partition \(\Gamma' = \Gamma_1 \cup \Gamma_2\) into two sets of \(2\) points each.
(This partition is unordered, i.e., we do not label which set is \(\Gamma_1\) and which set is \(\Gamma_2\).)
Given such a partition,
our goal is to construct a BN-curve \(C\) of degree \(7\) and genus \(4\) whose intersection with \(Q\) is \(\Gamma\).

Write \(L_i \subset \pp^3\) for the line joining the two points of \(\Gamma_i\).
By dimension count, our pencil of \((3, 3)\)-curves on \(Q\)
lifts uniquely to a pencil of cubic surfaces in \(\pp^3\)
containing \(L_1 \cup L_2\).
We construct \(C\) as the residual to \(L_1 \cup L_2\) in the base locus of this pencil of cubic surfaces.
By the liason formula, \(C\) has degree \(7\) and genus \(4\),
and is automatically a BN-curve because \(d > 2g - 2\).
Finally, by construction, \(C \cap Q = \Gamma\).

\bibliographystyle{amsplain.bst}
\bibliography{iqbib.bib}
\end{document}